\documentclass{article}

\usepackage{amsmath}
\usepackage{amsthm}

\title{The $q$-golden ratio, Catalan numbers, and an identity of Sauermann--Wigderson}
\author{Kevin Carde}
\date{2021-04-07}

\newtheorem{theorem}{Theorem}
\newtheorem{definition}[theorem]{Definition}
\newtheorem{proposition}[theorem]{Proposition}
\newtheorem{corollary}[theorem]{Corollary}

\begin{document}
\maketitle

\begin{abstract}
In this note, we present some basic properties of $q$-Fibonacci numbers
and their relationship to the $q$-golden ratio and Catalan numbers.
We then use this relationship to give a short proof of a combinatorial identity.
\end{abstract}

\section{$q$-Fibonacci numbers}

\begin{definition}
The \emph{$q$-Fibonacci numbers} are defined recursively by $F_0(q) = F_1(q) = 1$, $F_n(q) = F_{n-1}(q) + qF_{n-2}(q)$.
\end{definition}
Note that the original Fibonacci numbers are retrieved by setting $q=1$.

These $q$-Fibonaccis have a nice non-recursive form.

\begin{proposition}
$$
F_n(q) = \sum_{k=0}^{\lfloor \frac{n}{2} \rfloor} \binom{n-k}{k}q^k.
$$
\end{proposition}
\begin{proof}
We carry out a quick proof by induction. Clearly this holds for $F_0(q) = F_1(q) = 1$.
Assuming the statement for $n-1$ and $n-2$, we have
\begin{align*}
F_n(q) = F_{n-1}(q) + qF_{n-2}(q) &= \sum_k \left(\binom{n-1-k}{k} + \binom{n-2-(k-1)}{k-1}\right) q^k \\
 &= \sum_{k=0}^{\lfloor \frac{n}{2} \rfloor} \binom{n-k}{k}q^k
\end{align*}
as desired.
\end{proof}

From this, by plugging in $q=1$,
we immediately have an identity relating the standard Fibonacci numbers to binomial coefficients:
$$
F_n = \binom{n}{0} + \binom{n-1}{1} + \dots + \binom{\lceil\frac{n}{2}\rceil}{\lfloor\frac{n}{2}\rfloor}.
$$

\section{The $q$-golden ratio}

It is well known that the the ratio of successive Fibonacci numbers approaches the golden ratio $\varphi = \frac{1+\sqrt{5}}{2}$.
We prove the following $q$-analog:

\begin{theorem}
For $n \ge 0$, we have as a power series around $q=0$
$$
\frac{F_{n+1}(q)}{F_n(q)} = \frac{1+\sqrt{1+4q}}{2} + O(q^{n+1}).
$$
In particular, as $n\to\infty$, the ratio approaches $\varphi(q) = \frac{1+\sqrt{1+4q}}{2}$,
which itself equals the golden ratio $\varphi$ when $q=1$.
\end{theorem}

\begin{proof}

Before we begin, note that the theorem is equivalent (by taking reciprocals) to
$$
\frac{F_n(q)}{F_{n+1}(q)} = -\frac{1-\sqrt{1+4q}}{2q} + O(q^{n+1}).
$$

The theorem itself will be a straightforward proof by induction on $n$.
The base case of $n=0$ is immediate. We will now show that if the statement is true for $n-1$, it is also true for $n$.

By the defining recurrence for the $q$-Fibonaccis, we have
$$
\frac{F_{n+1}(q)}{F_n(q)} = \frac{F_n(q)+qF_{n-1}(q)}{F_n(q)} = 1 + q\frac{F_{n-1}(q)}{F_n(q)}.
$$
By the inductive hypothesis and the equivalent reciprocal formulation,
this is equal to
$$
1 + q\left(-\frac{1-\sqrt{1+4q}}{2q} + O(q^n)\right) = \frac{1+\sqrt{1+4q}}{2} + O(q^{n+1}),
$$
as desired.
\end{proof}

The formula appearing in this theorem requires no guessing to discover.
Ignoring the error term, the recurrence for the $q$-Fibonaccis tells us that
the $q$-golden ratio $\varphi(q)$ should satisfy $\varphi(q) = 1 + \frac{q}{\varphi(q)}$ or $\varphi(q)^2 - \varphi(q) - q = 0$,
the $q$-analog of the standard quadratic satisfied by $\varphi$.
By the quadratic formula and appropriate choice of signs, we immediately get $\varphi(q) = \frac{1+\sqrt{1+4q}}{2}$.
This all parallels how one might discover and prove that $F_{n+1}/F_n \to \varphi$ for the standard Fibonacci numbers.

\section{Catalan numbers and the Sauermann--Wigderson identity}

In the proof, we took the reciprocal
$$
\frac{F_n(q)}{F_{n+1}(q)} = -\frac{1-\sqrt{1+4q}}{2q} + O(q^{n+1}).
$$
Ignoring the error term, the right hand side is, upon replacing $q$ by $-q$,
the generating function for the Catalan numbers.

We therefore have the following:

\begin{corollary}
The coefficients on $1, q, q^2, \dots, q^n$ of $F_n(q)/F_{n+1}(q)$ are, up to sign, the Catalan numbers: $1, -1, 2, -5, 14, \dots, (-1)^nC_n$.
\end{corollary}

As a further corollary, we get a stronger version of an identity needed by Sauermann--Wigderson in \cite{SW}.
(This strengthening was noted and proved differently in \cite{Z}.)

\begin{corollary}
Let $m \le n$. Then
$$
\sum_{(m_1,\dots,m_t)} (-1)^t\binom{n-m_1}{m_1-1}\binom{n-m_2}{m_2}\cdots\binom{n-m_t}{m_t} = (-1)^m C_{m-1},
$$
where the sum is over all tuples (of varying length) of positive integers summing to $m$.
\end{corollary}
\begin{proof}
Note that
$$
\sum_k \binom{n-k}{k-1}q^k = \sum_k \binom{n-1-(k-1)}{k-1} q^k = qF_{n-1}(q).
$$

The other binomial coefficients appearing in the original identity are of the form $\binom{n-k}{k}$,
which are directly the coefficients in $F_{n}(q) - 1$.
(The requirement of positive $k$ means we need to subtract off the constant term.)

The desired quantity, therefore, is the $q^m$ coefficient in
$$
-qF_{n-1}(q)\Big(1 - (F_n(q)-1) + (F_n(q)-1)^2 - \dots\Big) = -q \frac{F_{n-1}(q)}{F_n(q)}.
$$

Because of the $-q$ out front, we are seeking the negative of the $q^{m-1}$ coefficient in $\frac{F_{n-1}(q)}{F_n(q)}$.
By the $q$-golden ratio Catalan corollary, our final result is $-(-1)^{m-1} C_{m-1} = (-1)^m C_{m-1}$, as desired.

\end{proof}

\end{document}